\documentclass[captions=tableabove]{scrartcl}

\usepackage{amsmath,amsfonts,amssymb,amsthm}
\usepackage{abstract}
\usepackage{enumerate}
\usepackage{longtable}
\usepackage{booktabs}
\usepackage{graphicx}
\usepackage{caption,subcaption}
\usepackage{diagbox}
\usepackage{multirow}
\usepackage{authblk}
\usepackage{cite}

\theoremstyle{plain}
\newtheorem{prop}{Proposition}
\newtheorem{thm}[prop]{Theorem}

\theoremstyle{definition}
\newtheorem{defn}[prop]{Definition}

\usepackage{todonotes}

\title{Fisher--Rao Geometry and Jeffreys Prior for Pareto Distribution}
\date{}
\author[1,2]{Mingming Li\thanks{3120160587@bit.edu.cn}}
\author[1,2]{Huafei Sun\thanks{huafeisun@bit.edu.cn}}
\author[3,1]{Linyu Peng\thanks{l.peng@aoni.waseda.jp}}
\affil[1]{ School of Mathematics and Statistics,  Beijing Institute of Technology, Beijing 100081, China}
\affil[2]{Beijing Key Laboratory on MCAACI, Beijing 100081, China}
\affil[3]{Waseda Institute for Advanced Study, Waseda University, Tokyo 169-8050, Japan}

\begin{document}

\maketitle
\begin{abstract}
In this paper, we investigate the Fisher--Rao geometry of the two-parameter family of Pareto distribution.
We prove that its geometrical structure is isometric to the Poincar\'e upper half-plane model, and then study the corresponding geometrical features by presenting explicit expressions for connection, curvature and geodesics.
It is then applied to Bayesian inference by considering the Jeffreys prior determined by the volume form. In addition, the posterior distribution from the prior is computed, providing a systematic method to the Bayesian inference for Pareto distribution.
\end{abstract}

\section{Introduction}

A statistical model is often described by a family of distributions which are indexed by a set of parameters. These parameters form a space which can be endowed with some differential-geometrical structures reflecting the properties of the specified distributions. These geometrical structures then provide geometrical approaches to deal with statistical problems. One of the most useful structures is the Fisher--Rao metric which stems from the use of the Fisher information matrix~\cite{Fisher1922}. Rao~\cite{Rao1945} considered the Fisher information from a differential-geometrical viewpoint. Gradually, the Fisher--Rao metric becomes one of the central concepts in the subject of information geometry, e.g.~\cite{Amari1985,Amari2007,Sun2016,Efron1975}, which combines ideas from differential geometry and information theory to study the geometrical structure of statistical models.

In this paper, we focus on the two-parameter family of Pareto distribution, a family of statistical models with power-law probability distributions that is often used in describing many scientific and social phenomena, e.g.~\cite{Arnold1983a,Zipf1949}.
Pareto distribution does not belong to the well-studied regular type of distributions, since the support of the probability density depends on one of its parameters. Consequently, the Fisher--Rao geometry for Pareto distribution is not as regular as many other distributions. Note that the Fisher--Rao metric is no longer equal to the negative Hessian form as it is expected in regular cases. Although both of them are symmetric, the negative Hessian form is not guaranteed to be positive definite on the whole parameter space. However, in the earlier work~\cite{Peng2007}, the geometrical structure of Pareto distribution was actually calculated by the negative Hessian form. Thus, this paper is first devoted to presenting the proper metric structure for Pareto distribution. Interestingly, this structure is readily identified, as we shall prove that it is isometric to the Poincar\'e upper half-plane model which is reminiscent of a similar result about the two-parameter family of one-dimensional normal distribution. Based on this observation, the geometrical characteristics of Pareto distribution such as curvature and geodesics can be readily obtained.
These results are included in Section~\ref{sec:geom}.

To illustrate an application of the geometrical structure for statistical inference over Pareto distribution, we utilize the Jeffreys prior to develop a systematic approach to Bayesian inference for Pareto distribution.
The Jeffreys prior (cf. \cite{Jeffreys1946}) is a non-informative prior distribution for a parameter space, that is, a prior distribution without any subjective information assumed. The basic idea of the Jeffreys prior is, to define a prior distribution such that the probability of finding a parameter in a specified region is proportional to the geometrical volume of this region. Different from the previous work \cite{KimKangLee2009},  no change of variables is introduced in this paper as a change of variable involving parameters may greatly affect the statistical and geometric properties of the distribution. For instance, it is well known that any normal distribution can be transformed to a standard one.
In this paper, our derived prior is an improper prior which does not have a proper probability distribution, nevertheless we can proceed to calculate the posterior distribution from the prior, and to obtain a proper posterior probability distribution as we shall show in Section~\ref{sec:bayes}. Then the posterior distribution can be directly applied to Bayesian inference for Pareto distributions, as an illustration of that, we present simulation results in Section~\ref{sec:sim}.

\section{Preliminaries}
For convenience, we review some necessary differential-geometrical concepts and results (cf.~\cite{Petersen2006,Tu2017}) which are used later in Section~\ref{sec:geom}.
\begin{defn}
\label{def:rie}
A \emph{Riemannian metric} on a smooth manifold $M$ is a $C^\infty$ assignment to each point $p\in M$ of an inner product $g_p$ on the tangent space $T_pM$. A \emph{Riemannian manifold} is a pair $(M,g)$ consisting of a manifold $M$ together with a Riemannian metric $g$ on $M$.
\end{defn}
\begin{defn}
A diffeomorphism $F:(M,g)\to(N,h)$ is called an \emph{isometry} between two Riemannian manifolds $(M,g)$ and $(N,h)$ if $F^*h=g$.
\end{defn}
\begin{thm}
On a Riemannian manifold there always exists a unique Riemannian connection, namely, an affine connection that is torsion-free and compatible with the metric.
\end{thm}
\begin{defn}
Let $\nabla$ be an affine connection on an $m$-dimensional manifold $M$
and let $e_1,\dotsc,e_m$ be a local frame on $M$.
For a vector field $X$, the \emph{connection forms} $\omega_j^i$ are defined by
\[
\nabla_Xe_j=\omega_j^i(X)e_i,
\]
and the matrix $[\omega_j^i]$ is called the \emph{connection matrix} of the connection $\nabla$ relative to the frame $e_1,\dotsc,e_m$.
Similarly, the \emph{curvature forms} $\Omega_j^i$ are defined by
\[
R(X,Y)e_j=\Omega_j^i(X,Y)e_i,
\]
where the curvature tensor $R$ is given by
\[
R(X,Y)=\nabla_X\nabla-\nabla_Y\nabla_X-\nabla_{[X,Y]},
\]
and the matrix $[\Omega_j^i]$ is called the \emph{curvature matrix} of the connection $\nabla$ relative to the frame $e_1,\dotsc,e_m$.
\end{defn}
Note that we adopt the Einstein summation convention here and throughout the paper.
\begin{prop}
The curvature forms $\Omega_j^i$ are related to the connection forms $\omega_j^i$ by the second structural equation:
\begin{equation}
\label{eqn:2se}
\Omega_j^i=\mathrm d\omega_j^i+\omega_k^i\wedge\omega_j^k.
\end{equation}
\end{prop}
\begin{prop}
\label{prop:skew}
Let $(M,g)$ be an $m$-dimensional Riemannian manifold, $\nabla$ be the Riemannian connection,  $e_1,\dotsc,e_m$ be an orthonormal frame, and $\theta^1,\dotsc,\theta^m$ be the dual frame.
Then the connection matrix $[\omega_j^i]$ of $\nabla$ relative to $e_1,\dotsc,e_m$ is a skew symmetric matrix such that the first structural equation holds,
\begin{equation}
\label{eqn:1se}
\mathrm d\theta^i+\omega_j^i\wedge\theta^j=0,\ \forall i=1,\dotsc,m.
\end{equation}
Consequently, the curvature matrix $[\Omega_j^i]$ is also skew symmetric.
\end{prop}
\begin{prop}
Let $(U,x^1,\dotsc,x^m)$ be a coordinate chart on a Riemannian manifold $(M,g)$, and $g_{ij}=g\left(\frac{\partial}{\partial x^i},\frac{\partial}{\partial x^j}\right)$. Then the volume form of $M$ on $U$ is given by
\begin{equation}
\label{eqn:volm}
\mathrm{vol}=\sqrt{\det [g_{ij}]}\,\mathrm dx^1\wedge\dotsm\wedge\mathrm dx^m.
\end{equation}
\end{prop}
\begin{defn}
\label{def:chr}
Let $\nabla$ be an affine connection on a manifold $M$
and $\partial_1,\dotsc,\partial_m$ be a local coordinate frame on $M$.
Then the \emph{Christoffel symbols} $\Gamma_{ij}^k$ of $\nabla$ relative to $\partial_1,\dotsc,\partial_m$ are defined by
\[
\nabla_{\partial_i}\partial_j=\Gamma_{ij}^k\partial_k.
\]
\end{defn}
\begin{prop}
Let $(U,x^1,\dotsc,x^m)$ be a coordinate chart on a manifold with
$\Gamma_{ij}^k$ be the Christoffel symbols of a connection.
Then the geodesic equations are given by
\begin{equation}
\label{eqn:geod}
\ddot x^k+\Gamma_{ij}^k\dot x^i\dot x^j=0,\ k=1,\dotsc,m.
\end{equation}
\end{prop}

\section{Geometry of the two-parameter family of Pareto distribution}
\label{sec:geom}
In this section, we study the geometrical structure related to the Fisher--Rao metric of the two-parameter family of Pareto distribution.
The Fisher--Rao metric provides us with a Riemannian manifold structure and makes the differential-geometrical tools applicable on the parameter space of statistical models.

For a family of distribution with probability function $p(x\,|\,\theta)$, where $\theta=(\theta^1,\dotsc,\theta^m)\in\Theta$  being an open subset of $\mathbb R^m$, the Fisher--Rao metric is defined below.
\begin{defn}\label{defn:FR}
Let $X$ denote a random variable which represents the probability function $p(x\,|\,\theta)$ and $\partial_i=\frac{\partial}{\partial\,\theta^i}$.
The Fisher--Rao metric matrix about the frame $\partial_1,\dotsc,\partial_n$ is defined via the expectation as
\begin{equation}
\label{eqn:fr}
g_{ij}(\theta)=E[\partial_il(X\,|\,\theta)\,\partial_jl(X\,|\,\theta)]=\int\partial_il(x\,|\,\theta)\partial_jl(x\,|\,\theta)p(x\,|\,\theta)\mathrm dx,
\end{equation}
where $l(x\,|\,\theta)=\log p(x\,|\,\theta)$ is the log-likelihood function.
\end{defn}

Suppose that the following {\em regularity conditions} hold, namely,
\begin{enumerate}[(i)]
\item For each $x$,
the mapping $\theta\mapsto p(x\,|\,\theta)$ is smooth.
\item The order of integration and differentiation can be freely
rearranged. For instance,
\begin{equation}
\int \partial_ip(x\,|\,\theta)\operatorname{d}\!x=
\partial_i\int p(x\,|\,\theta)\operatorname{d}\!x=\partial_i 1=0.
\end{equation}
For discrete distributions, we simply replace the integration by summation.
\item Different parameters stand for different probability density functions, that is,
$\theta_1\neq\theta_2$ implies that $p(x\,|\,\theta_1)$ and
$p(x\,|\,\theta_2)$ are different. Moreover, every parameter
$\theta$ possesses a common support where $p(x\,|\,\theta)>0$.
\end{enumerate}
Then we have
\begin{equation}
\label{eqn:zero}
E[\partial_il(X\,|\,\theta)]=0
\end{equation}
and the Fisher--Rao metric in the negative Hessian form
\begin{equation}
\label{eqn:alt}
g_{ij}(\theta)=-E[\partial_i\partial_jl(X\,|\,\theta)].
\end{equation}

Now, we calculate the Fisher--Rao metric for the two-parameter family of Pareto distribution.
Its probability density function is given by
\begin{equation}
\label{eqn:pdf}
p(x\,|\,\alpha,\beta)=\frac{\beta\alpha^\beta}{x^{\beta+1}}I_{[x\ge\alpha]},\ \alpha>0, \beta>0.
\end{equation}
Thus, the log-likelihood function is given by
\[
l(x\,|\,\alpha,\beta)=\log p(x\,|\,\alpha,\beta)=\log\beta+\beta\log\alpha-(\beta+1)\log x.
\]

Note that, with parameters $\theta=(\alpha,\beta)\in\Theta=\mathbb R^+\times\mathbb R^+$, the probability density function of Pareto distribution does not satisfy the second regularity condition since  the support of  $x$ depends on parameter $\alpha$,
and hence the negative Hessian form {\eqref{eqn:alt} is not valid as the Fisher--Rao metric.
If it is used regardless,
one would obtain a `fake' metric matrix
\[
\left(
 \begin{array}{rr}
  \frac{\beta}{\alpha^2} &-\frac{1}{\alpha} \\
   -\frac{1}{\alpha} & \frac{1}{\beta^2}
 \end{array}
\right) ,
\]
which is not positive definite unless $0<\beta<1$.

Set $\theta^1=\alpha$, $\theta^2=\beta$ and $\partial_1=\partial/\partial\alpha$, $\partial_2=\partial/\partial\beta$.
By Definition \ref{defn:FR} and the observation that the random variable $(\log X-\log\alpha)$ obeys the exponential distribution with mean value $1/\beta$,
the proper Fisher--Rao metric is calculated as
\begin{equation}
\label{eqn:gij}
g_{11}=\frac{\beta^2}{\alpha^2},\ g_{12}=g_{21}=0,\ g_{22}=\frac{1}{\beta^2}.
\end{equation}
Above results about the Fisher--Rao metric  of Pareto distribution are summarized as follows.
\begin{prop}
The tensor expression of the Fisher--Rao metric for the two-parameter family of Pareto distribution is given by
\begin{equation}
\label{eqn:metric}
g=\frac{\beta^2}{\alpha^2}\,\mathrm d\alpha\otimes\mathrm d\alpha+\frac{1}{\beta^2}\,\mathrm d\beta\otimes\mathrm d\beta.
\end{equation}
\end{prop}

In the rest of this paper, we shall denote the statistical manifold of Pareto distribution by
\[
P=\left\{\,p_{\alpha,\beta}\mid p_{\alpha,\beta}(x)=p(x\,|\,\alpha,\beta),\ \alpha>0, \beta>0\right\}.
\]
Together with the Fisher--Rao metric $g$ given by~\eqref{eqn:metric}, $(P,g)$ becomes a Riemannian manifold in the sense of Definition~\ref{def:rie}.

\subsection{An isometry between the Pareto manifold and the Poincar\'e upper half-plane model}
In this subsection, we shall show that the Riemannian manifold $(P,g)$ is isometric to the Poincar\'e upper half-plane model.

The Poincar\'e upper half-plane model (cf. \cite{Jost2006,Stahl1993}) is the upper half-plane
\[
H=\left\{(x,y)\in\mathbb R^2 \mid y>0\right\},
\]
together with the Poincar\'e metric
\[
h=\frac{\mathrm dx\otimes\mathrm dx+\mathrm dy\otimes\mathrm dy}{y^2}.
\]
Let $F:(P,g)\to(H,h)$ map $p_{\alpha,\beta}$ to $(\log\alpha,1/\beta)$.
Then, we have
\begin{align*}
F^*h&=
\frac{\mathrm d\log\alpha\otimes\mathrm d\log\alpha+\mathrm d(1/\beta)\otimes\mathrm d(1/\beta)}{(1/\beta)^2}\\
&=\frac{(1/\alpha)^2\mathrm d\alpha\otimes\mathrm d\alpha+(1/\beta)^4\mathrm d\beta\otimes\mathrm d\beta}{(1/\beta)^2}\\
&=\frac{\beta^2}{\alpha^2}\,\mathrm d\alpha\otimes\mathrm d\alpha+\frac{1}{\beta^2}\,\mathrm d\beta\otimes\mathrm d\beta=g,
\end{align*}
which yields the following Proposition~\ref{prop:hyp}.
\begin{prop}
\label{prop:hyp}
The diffeomorphism $F$ defined above is an isometry between $(P,g)$ and $(H,h)$.
\end{prop}
Now we can make use of the geometry of the Poincar\'e upper half-plane model to study the statistical manifold of Pareto distribution,
since isometry preserves  essential geometrical structures.

\subsection{Connection form, curvature form and Christoffel symbols}
As a consequence of Proposition~\ref{prop:hyp}, $(P,g)$ has constant Gaussian curvature $K=-1$. It hence contributes as another member of the statistical manifolds with constant curvatures, e.g. \cite{Peng2019,Cao2008}.

Now we shall study the geometrical structure of $(P,g)$ in detail by using differential forms,
which have been tools of great power and versatility in differential geometry,
since \'Elie Cartan pioneered its use in the 1920s~\cite{Cartan1952-1955}.
In this subsection, we derive the connection and curvature for $(P,g)$ in terms of differential forms.

With the metric given by~\eqref{eqn:metric}, we obtain an orthonormal frame as
\begin{equation}
\label{eqn:frame}
e_1=\frac{\alpha}{\beta}\,\partial_1,\ e_2=\beta\,\partial_2.
\end{equation}
The dual frame with respect to~\eqref{eqn:frame} is given by
\begin{equation}
\label{eqn:dual}
\theta^1=\frac{\beta}{\alpha}\,\mathrm d\alpha,\ \theta^2=\frac{1}{\beta}\,\mathrm d\beta.
\end{equation}
The volume form is given by
\begin{equation}
\label{eqn:vol}
\mathrm{vol}=\theta^1\wedge\theta^2=\frac{1}{\alpha}\,\mathrm d\alpha\wedge\mathrm d\beta,
\end{equation}
which can also be derived from \eqref{eqn:volm}.

Let $\nabla$ be the unique Riemannian connection on $(P,g)$.
Let $[\omega_j^i]$ and $[\Omega_j^i]$ be the connection and curvature matrices of $\nabla$ relative to $e_1,e_2$, respectively.
By the skew-symmetry stated in Proposition~\ref{prop:skew}, we only need to determine $\omega_2^1$ and $\Omega_2^1$ for $(P,g)$.

By differentiating \eqref{eqn:dual}, we have
\begin{equation}
\label{eqn:diff1}
\mathrm d\theta^1=-\frac{1}{\alpha}\,\mathrm d\alpha\wedge\mathrm d\beta,\ \mathrm d\theta^2=0.
\end{equation}
The first structural equation~\eqref{eqn:1se} reads
\begin{equation}
\label{eqn:diff2}
\mathrm d\theta^1=-\omega_2^1\wedge\theta^2,\ \mathrm d\theta^2=\omega_2^1\wedge\theta^1.
\end{equation}
By comparing~\eqref{eqn:diff1} and~\eqref{eqn:diff2}, we obtain
\[
\omega_2^1=\frac{\beta}{\alpha}\,\mathrm d\alpha.
\]
Using the second structural equation~\eqref{eqn:2se}, we get
\[
\Omega_2^1=\mathrm d\omega_2^1=-\frac{1}{\alpha}\,\mathrm d\alpha\wedge\mathrm d\beta=K\mathrm{vol}.
\]

Now we present another description of the connection $\nabla$ by the Christoffel symbols $\Gamma_{ij}^k$ defined in Definition~\ref{def:chr}.

For any smooth vector field $X$ on $M$, we have
\begin{align*}
\nabla_X\partial_1&=\nabla_X[(\beta/\alpha)e_1]=X(\beta/\alpha)e_1+(\beta/\alpha)\nabla_Xe_1\\
&=\left((X\beta)/\alpha-\beta (X\alpha)/\alpha^2\right)(\alpha/\beta)\partial_1-(\beta/\alpha)\,\omega_2^1(X)e_2\\
&=\left(\frac{X\beta}{\beta}-\frac{X\alpha}{\alpha}\right)\,\partial_1-\frac{\beta^3}{\alpha^2}(X\alpha)\,\partial_2.
\end{align*}
Similarly, we have
\begin{align*}
\nabla_X\partial_2&=\nabla_X[(1/\beta)e_2]=X(1/\beta)e_2+(1/\beta)\nabla_Xe_2\\
&=-[(1/\beta)^2X\beta]\beta\partial_2+(1/\beta)\,\omega_2^1(X)e_1\\
&=\frac{X\alpha}{\beta}\,\partial_1-\frac{X\beta}{\beta}\,\partial_2.
\end{align*}
Hence, relative to the coordinate frame $\partial_1,\partial_2$, we obtain
\begin{align*}
\nabla_{\partial_1}\partial_1&=-\frac{1}{\alpha}\,\partial_1-\frac{\beta^3}{\alpha^2}\,\partial_2,
&\nabla_{\partial_2}\partial_1&=\frac{1}{\beta}\,\partial_1,\\
\nabla_{\partial_1}\partial_2&=\frac{1}{\beta}\,\partial_1,
&\nabla_{\partial_2}\partial_2&=-\frac{1}{\beta}\,\partial_2.
\end{align*}
The corresponding Christoffel symbols $\Gamma_{ij}^k$ are given in Table~\ref{tab:chr}.
\begin{table}[htbp]
\caption{Christoffel symbols}
\label{tab:chr}
\centering
\begin{tabular}{l|cc}
\diagbox{$ij$}{$k$}&1&2\\
\hline
11&$-1/\alpha$&$-\beta^3/\alpha^2$\\
12&$1/\beta$&0\\
21&$1/\beta$&0\\
22&0&$-1/\beta$
\end{tabular}
\end{table}
\subsection{Geodesics and geodesic distances}
By~\eqref{eqn:geod} and Table~\ref{tab:chr}, we obtain the geodesic equations for $(P,g)$ as
\begin{equation}
\label{eqn:geoeq}
\ddot{\alpha}-\frac{\dot{\alpha}^2}{\alpha}+\frac{2\dot{\alpha}\dot{\beta}}{\beta}=0,\
\ddot{\beta}-\frac{\beta^3\dot{\alpha}^2}{\alpha^2}-\frac{\dot{\beta}^2}{\beta}=0.
\end{equation}

Thanks to the well-known results about $(H,h)$, we can derive explicit expression of the geodesics on $(P,g)$ by Proposition~\ref{prop:hyp} instead of solving~\eqref{eqn:geoeq} directly.
On $(H,h)$, the unit-speed geodesic starting from $(x_0,y_0)\in H$, with the initial velocity making an angle $\theta_0$ about the positive $x$-axis, is given by
\begin{align*}
x(t)&=x_0+y(t)\sinh t\cos\theta_0,\\
y(t)&=\frac{y_0}{\mathrm e^t\sin^2(\pi/4-\theta_0/2)+\mathrm e^{-t}\cos^2(\pi/4-\theta_0/2)}.
\end{align*}
Hence, the corresponding geodesic starting from $(\alpha_0,\beta_0)$ on $(P,g)$ is given by $F^{-1}(x(t),y(t))$, i.e.,
\begin{equation}
\begin{aligned}
\label{eqn:geoab}
\alpha(t)&=\alpha_0\exp\left(\frac{\sinh t\cos\theta_0}{\beta(t)}\right),\\
\beta(t)&=\beta_0\left(\mathrm e^t\sin^2(\pi/4-\theta_0/2)+\mathrm e^{-t}\cos^2(\pi/4-\theta_0/2)\right).
\end{aligned}
\end{equation}
Fig.~\ref{fig:geo} illustrates the radial geodesics given by \eqref{eqn:geoab} with $(\alpha_0,\beta_0)=(1,1)$ and $\theta_0=k\pi/16,k=0,1,\dotsc,31$ respectively, which outline the shape of the unit geodesic ball with $t\in[0,1]$.
\begin{figure}[htbp]
\centering
  \includegraphics[scale=0.33]{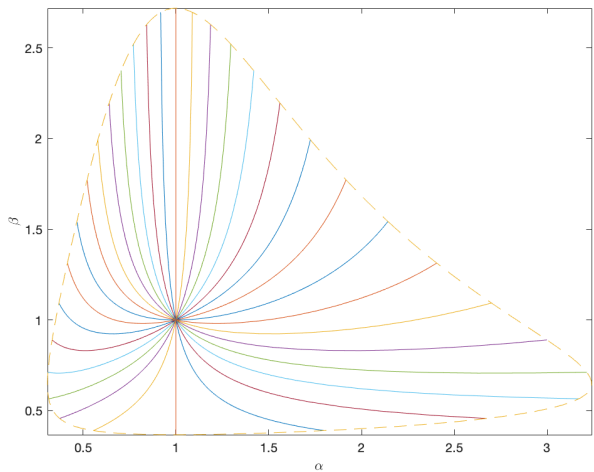}
\caption{Geodesic ball centered at $(1,1)$}
\label{fig:geo}
\end{figure}

On $(H,h)$, the geodesic distance between two points $(x_0,y_0)$ and $(x_1,y_1)$ is given by
\[
d_H((x_0,y_0),(x_1,y_1))=\mathrm{arcosh}\left(1+\frac{(x_0-x_1)^2+(y_0-y_1)^2}{2y_0y_1}\right).
\]
Hence, the geodesic distance on $(P,g)$ between probability densities $p_{\alpha_0,\beta_0}$ and $p_{\alpha_1,\beta_1}$ is given by $d_H(F(p_{\alpha_0,\beta_0}),F(p_{\alpha_1,\beta_1}))$, i.e.,
\begin{equation}
\label{eqn:dist}
d(p_{\alpha_0,\beta_0},p_{\alpha_1,\beta_1})=\mathrm{arcosh}\,\left(1+\frac{\beta_0\beta_1(\log\alpha_0-\log\alpha_1)^2}{2}+\frac{(\beta_0-\beta_1)^2}{2\beta_0\beta_1}\right).
\end{equation}

\section{An application to Jeffreys prior}
\label{sec:bayes}
In statistics, the first step in Bayesian inference for parametric models is to select an appropriate prior distribution for the related parameters.
As there already exist many useful results about Bayesian inference for Pareto distribution, e.g.~\cite{Arnold1983,Arnold1989},
we shall focus ourselves on the Bayesian approach generated from the so-called Jeffreys prior.
The Jeffreys prior is a non-informative prior which is directly related to the Fisher--Rao metric by stipulating that the prior probability be proportional to the geometrical volume in the parameter space.
Although there already exist lots of priors from which one can carefully choose for practical use,
our choice of the Jeffreys prior is made here in order to observe the statistical feature of our specified geometrical structure.

By the volume form in~\eqref{eqn:vol}, we obtain the Jeffreys prior for the two-parameter Pareto model as
\begin{equation}
\label{eqn:prior}
p(\alpha,\beta)\propto \frac{1}{\alpha},\,\alpha>0,\beta>0.
\end{equation}
The proportionality in~\eqref{eqn:prior} cannot give us a proper probability distribution
as integration of $p(\alpha,\beta)$ yields infinity,
but an improper prior can still be useful.
Improper priors usually provide much less information than that available in the observed data,
which is a desired property as most inference problems are mainly based on data analysis rather than the specification of priors.

To deal with the inference problem after we observed the data $\mathbf x=(x_1,\dotsc,x_n)$ from Pareto distribution with unknown parameters,
we first calculate the posterior distribution by using~\eqref{eqn:prior} as the improper prior.
Bayes' theorem gives the posterior probability density as
\begin{equation}
\label{eqn:bayes}
p(\alpha,\beta\,|\, \mathbf x)=\frac{p(\mathbf x\,|\, \alpha,\beta)\, p(\alpha,\beta)}{\int_0^\infty\int_0^\infty p(\mathbf x\,|\,\alpha,\beta)\, p(\alpha,\beta)\, \mathrm d\alpha\, \mathrm d\beta},
\end{equation}
where the joint probability density of the observations $\mathbf x$ is given by
\begin{equation}
\label{eqn:joint}
p(\mathbf x\,|\,\alpha,\beta)=\beta^n\alpha^{n\beta}{\left(\prod_{i=1}^nx_i\right)}^{-\beta-1}I_{[\min_{i=1}^nx_i\ge\alpha]}.
\end{equation}
By the factorization criterion, we note from~\eqref{eqn:joint} that the statistics
\[
q_1(\mathbf x)=\min_{i=1}^nx_i,\ q_2(\mathbf x)=\sum_{i=1}^n\log x_i
\]
are jointly sufficient statistics for the parameters $\alpha,\beta$.
Actually, by classical results in statistics (cf.~\cite{Casella2001,DeGroot2012}),
$q_1(\mathbf x),q_2(\mathbf x)$ are minimal jointly sufficient statistics for $\alpha,\beta$,
since they are equivalent to the maximum likelihood estimator (MLE) of $\alpha,\beta$:
\[
\widehat{\alpha}(\mathbf x)=q_1(\mathbf x),\ \widehat{\beta}(\mathbf x)=\frac{n}{q_2(\mathbf x)-n\log q_1(\mathbf x)}.
\]
By direct computation, we have
\begin{align*}
\int_0^\infty\int_0^\infty p(\mathbf x\,|\, \alpha,\beta)\, p(\alpha,\beta)\mathrm d\alpha\, \mathrm d\beta
&=\int_0^\infty\beta^n\exp\left[-q_2(\mathbf x)(\beta+1)\right]\int_0^{q_1(\mathbf x)} \alpha^{n\beta-1}\mathrm d\alpha\, \mathrm d\beta\\
&=\frac{1}{n}\int_0^\infty\beta^{n-1}q_1^{n\beta}(\mathbf x)\exp\left[-q_2(\mathbf x)(\beta+1)\right]\mathrm d\beta\\
&=\frac{\exp\left[-q_2(\mathbf x)\right]\Gamma(n)}{n\left[q_2(\mathbf x)-n\log q_1(\mathbf x)\right]^{n}}.
\end{align*}
Thus, by~\eqref{eqn:bayes}, we obtain the joint posterior probability density for $\alpha,\beta$ as
\begin{equation}
\label{eqn:post}
p(\alpha,\beta\,|\, \mathbf x)=\frac{n\left[q_2(\mathbf x)-n\log q_1(\mathbf x)\right]^n}{\Gamma(n)}\beta^n\alpha^{n\beta-1}\exp\left[-q_2(\mathbf x)\beta\right]I_{[0<\alpha\le q_1(\mathbf x)]}.
\end{equation}

Next, we calculate the marginal posteriors by integrating~\eqref{eqn:post}. The marginal posterior of $\alpha$ is given by
\begin{equation}
\label{eqn:marga}
p(\alpha\,|\, \mathbf x)=\int_0^\infty p(\alpha,\beta\,|\,\mathbf x)\mathrm d\beta
=\frac{n^2\left[q_2(\mathbf x)-n\log q_1(\mathbf x)\right]^n}{\alpha\left[q_2(\mathbf x)-n\log\alpha\right]^{n+1}}I_{[0<\alpha\le q_1(\mathbf x)]},
\end{equation}
and the cumulative distribution function of $\alpha$ is
\begin{equation}
\label{eqn:cdf}
\Pr(\alpha\le t\,|\,\mathbf x)={\left(\frac{q_2(\mathbf x)-n\log q_1(\mathbf x)}{q_2(\mathbf x)-n\log t}\right)}^n,\ 0<t\le q_1(\mathbf x).
\end{equation}
We notice from~\eqref{eqn:marga} that one anomalous posterior behavior of $\alpha$ is that $p(\alpha\,|\, \mathbf x)\to\infty$ as $\alpha\to0^+$,
and this phenomenon cannot be eliminated even with increasing sample size $n$.
However, from~\eqref{eqn:cdf}, we observe that
\begin{equation}
\label{eqn:cdfa}
\Pr(\alpha\le t\,|\,\mathbf x)=\frac{1}{\left[1+\widehat{\beta}(\mathbf x)\log(\widehat{\alpha}(\mathbf x)/t)\right]^n},
\end{equation}
which tends to 0 for $t<c\widehat{\alpha}(\mathbf x)$ with positive $c<1$ as $n\to\infty$.
This means that the distribution~\eqref{eqn:marga} is concentrated in the vicinity of $\widehat{\alpha}(\mathbf x)$ for large sample size $n$ in spite of the unboundedness near 0.
Similarly, the marginal posterior of $\beta$ is obtained as
\begin{equation}
\label{eqn:margb}
p(\beta\,|\,\mathbf x)
=\frac{\beta^{n-1}\left[q_2(\mathbf x)-n\log q_1(\mathbf x)\right]^n}{\Gamma(n)}\exp\left[-\left(q_2(\mathbf x)-n\log q_1(\mathbf x)\right)\beta\right],
\end{equation}
which is a probability density of gamma distribution.
Hence, the posterior feature of $\beta$ is readily identified as it belongs to a well-studied distribution.
Suppose the random variables $V_1,\dotsc,V_n$ are independently drawn from the exponential distribution with population mean $\widehat{\beta}(\mathbf x)$, then the sample mean $\frac{1}{n}\sum_{i=1}^nV_i$ obeys the gamma distribution~\eqref{eqn:margb}.
As a consequence of the central limit theorem, for large sample size $n$, the distribution of $\beta$ in~\eqref{eqn:margb} approximates to the normal distribution with mean $\widehat{\beta}(\mathbf x)$ and variance $\widehat{\beta}^2(\mathbf x)/n$.

Furthermore, to deal with the situation when either of $\alpha$ and $\beta$ is known, we can derive the conditional posteriors from the joint posterior and the marginal posteriors.
The conditional posterior of $\alpha$ with known $\beta$ is given by
\[
p(\alpha\,|\,\mathbf x,\beta)=\frac{n\beta\alpha^{n\beta-1}}{q_1^{n\beta}(\mathbf x)}I_{[0<\alpha\le q_1(\mathbf x)]},
\]
with cumulative distribution function
\begin{equation}
\label{eqn:ccdf}
\Pr(\alpha\le t\,|\,\mathbf x,\beta)=[t/\widehat\alpha(\mathbf x)]^{n\beta},\ 0<t\le\widehat\alpha(\mathbf x).
\end{equation}
The conditional posterior of $\beta$ with known $\alpha$ is given by
\[
p(\beta\,|\,\mathbf x,\alpha)=\frac{\beta^n\left[q_2(\mathbf x)-n\log\alpha\right]^{n+1}}{\Gamma(n+1)}\exp\left[-\left(q_2(\mathbf x)-n\log\alpha\right)\beta\right],
\]
which is again a gamma distribution.
Both of these two distributions are easy to manipulate,
and we shall not discuss them in depth.

Now we can derive some estimators and predictions for the parameters $\alpha,\beta$ from the posterior distributions.
By setting $\Pr(\alpha\le t\,|\, \mathbf x)=1/2$ in~\eqref{eqn:cdfa}, the posterior median $\widetilde\alpha(\mathbf x)$ of $\alpha$ can be obtained as
\[
\widetilde\alpha(\mathbf x)=\widehat\alpha(\mathbf x)\exp\left(\frac{1-\sqrt[n]{2}}{\widehat\beta(\mathbf x)}\right).
\]
Similarly, by using~\eqref{eqn:ccdf}, the posterior median $\widetilde\alpha(\mathbf x,\beta)$ of $\alpha$ with known $\beta$ can be obtained as
\[
\widetilde\alpha(\mathbf x,\beta)=2^{-\frac{1}{n\beta}}\widehat\alpha(\mathbf x).
\]
As the posterior and conditional posterior of $\beta$ both satisfy the gamma distribution,
their medians do not have simple closed form.

The posterior mean $\bar\alpha(\mathbf x)$ of $\alpha$ can be computed as
\begin{equation}
\label{eqn:postm}
\bar\alpha(\mathbf x)=E(\alpha\,|\,\mathbf x)=\int_0^{\widehat{\alpha}(\mathbf x)}t\ \mathrm d\Pr(\alpha\le t\,|\,\mathbf x)
=\widehat{\alpha}(\mathbf x)-\int_0^{\widehat{\alpha}(\mathbf x)}\Pr(\alpha\le t\,|\,\mathbf x)\,\mathrm dt.
\end{equation}
Although it does not yield a simple closed form,
we can still obtain the following bounds for $\bar\alpha(\mathbf x)$.
\begin{prop} These inequalities hold for the posterior mean $\bar\alpha(\mathbf x)$:
\begin{equation}
\label{eqn:ineq}
\frac{(n-1)\widehat\beta(\mathbf x)-1}{(n-1)\widehat\beta(\mathbf x)}\widehat\alpha(\mathbf x)\le\bar\alpha(\mathbf x)\le\frac{n\widehat\beta(\mathbf x)}{n\widehat\beta(\mathbf x)+1}\widehat{\alpha}(\mathbf x).
\end{equation}
\end{prop}
\begin{proof}
Note that
\[
1+\widehat{\beta}(\mathbf x)\log\left(\widehat{\alpha}(\mathbf x)/t\right)\le\exp\left[\widehat{\beta}(\mathbf x)\log\left(\widehat{\alpha}(\mathbf x)/t\right)\right]
=\left[\widehat{\alpha}(\mathbf x)/t\right]^{\widehat{\beta}(\mathbf x)}.
\]
From this, we have
\[
\int_0^{\widehat{\alpha}(\mathbf x)}\Pr(\alpha\le t\,|\,\mathbf x)\,\mathrm dt\ge\int_0^{\widehat{\alpha}(\mathbf x)}\left[t/\widehat{\alpha}(\mathbf x)\right]^{n\widehat{\beta}(\mathbf x)}\mathrm dt=\frac{\widehat{\alpha}(\mathbf x)}{n\widehat\beta(\mathbf x)+1}.
\]
By~\eqref{eqn:postm}, we obtain the second inequality of~\eqref{eqn:ineq}.

By change of variables $u=\log(\widehat{\alpha}(\mathbf x)/t)$, we have
\begin{align*}
\int_0^{\widehat{\alpha}(\mathbf x)}\Pr(\alpha\le t\,|\,\mathbf x)\,\mathrm dt
&=\int_0^\infty\frac{\widehat{\alpha}(\mathbf x)\exp(-u)}{\left[1+\widehat{\beta}(\mathbf x)u\right]^n}\mathrm du\\
&\le\int_0^\infty\frac{\widehat{\alpha}(\mathbf x)}{\left[1+\widehat{\beta}(\mathbf x)u\right]^n}\, \mathrm du
=\frac{\widehat{\alpha}(\mathbf x)}{(n-1)\widehat\beta(\mathbf x)}.
\end{align*}
This yields the first inequality of~\eqref{eqn:ineq}.
\end{proof}
The other posterior means can be readily computed from the corresponding distributions. We present these results   without details as follows. The posterior mean of $\alpha$ with known $\beta$ is given by
\[
\bar\alpha(\mathbf x,\beta)=E(\alpha\,|\, \mathbf x,\beta)=\frac{n\beta}{n\beta+1}\widehat{\alpha}(\mathbf x).
\]
The posterior mean of $\beta$ with unknown $\alpha$ is given by
\[
\bar\beta(\mathbf x)=E(\beta\,|\,\mathbf x)=\widehat{\beta}(\mathbf x).
\]
The posterior mean of $\beta$ with known $\alpha$ is given by
\[
\bar\beta(\mathbf x,\alpha)=E(\beta\,|\,\mathbf x,\alpha)=\frac{n+1}{q_2(\mathbf x)-n\log\alpha}.
\]

We can also easily determine the posterior predictive distribution of a new observation $\widetilde x$ which is independently drawn from the Pareto distribution with the same parameters as the previous observations $\mathbf x$.
The posterior predictive distribution with unknown $\alpha,\beta$ is
\begin{equation}
\label{eqn:pred}
\begin{aligned}
p(\widetilde x\,|\,\mathbf x)&=\int_0^\infty\int_0^\infty p(\widetilde x\,|\,\alpha,\beta)\, p(\alpha,\beta\,|\,\mathbf x)\mathrm d\alpha\, \mathrm d\beta\\
&=\frac{n^2\left[q_2(\mathbf x)-n\log q_1(\mathbf x)\right]^{n}}{(n+1)\widetilde{x}\left[q_2(\mathbf x)+\log\widetilde x-(n+1)\log\min\{\widetilde x,q_1(\mathbf x)\}\right]^{n+1}}I_{[\widetilde x>0]}.
\end{aligned}
\end{equation}
The posterior predictive distribution with known $\alpha$ is
\begin{align*}
p(\widetilde x\,|\, \mathbf x,\alpha)&=\int_0^\infty p(\widetilde x\,|\, \alpha,\beta)\, p(\beta\,|\, \mathbf x,\alpha)\, \mathrm d\beta\\
&=\frac{(n+1)\left[q_2(\mathbf x)-n\log\alpha\right]^{n+1}}{\widetilde x\left[q_2(\mathbf x)+\log\widetilde x-(n+1)\log\alpha\right]^{n+2}}I_{[\widetilde x\ge\alpha]}.
\end{align*}
The posterior predictive distribution with known $\beta$ is
\begin{align*}
p(\widetilde x\,|\, \mathbf x,\beta)&=\int_0^\infty p(\widetilde x\,|\,\alpha,\beta)\, p(\alpha\,|\, \mathbf x,\beta)\, \mathrm d\alpha\\
&=\frac{n}{n+1}\beta\widehat{\alpha}^{-n\beta}(\mathbf x)\widetilde x^{-\beta-1}\min\{\widetilde x,\widehat{\alpha}(\mathbf x)\}^{(n+1)\beta}I_{[\widetilde x>0]}\\
&=
\begin{cases}
\frac{n}{n+1}\beta\widehat{\alpha}^{-n\beta}(\mathbf x)\widetilde x^{n\beta-1},&0<\widetilde x<\widehat{\alpha}(\mathbf x),\\
\frac{n}{n+1}\beta\widehat{\alpha}^{\beta}(\mathbf x)\widetilde x^{-\beta-1},&\widetilde x\ge\widehat{\alpha}(\mathbf x).
\end{cases}
\end{align*}

\section{Simulations}
\label{sec:sim}
We shall generate random samples for simulation of the Pareto distribution with underlying parameters fixed as $\alpha_0,\beta_0$. By the inverse transform sampling method, the desired samples from Pareto distribution can be generated as $X=\alpha_0U^{-1/\beta_0}$, where $U$ is drawn from random numbers uniformly distributed in the unit interval $(0,1)$.

\begin{table}[htbp]
\caption{Estimators for Pareto distribution}
\label{tab:est}
 \centering
 \begin{tabular}{l|l|ccc}
 \toprule
 \multicolumn{2}{c|}{}&$\alpha$&$\beta$&$d$\\
 \hline
 \multirow{3}*{unknown $\alpha,\beta$}&MLE &1.0303&1.1271&0.1238\\
 &Posterior median &1.0240&1.1234&0.1190\\
 &Posterior mean  &1.0211&1.1271&0.1217\\
 \hline
 \multirow{3}*{known $\alpha=\alpha_0$}&MLE &1&1.0905&0.0866\\
 &Posterior median &1&1.0977&0.0932\\
 &Posterior mean  &1&1.1014&0.0965\\
 \hline
 \multirow{3}*{known $\beta=\beta_0$}&MLE &1.0303&1&0.0298\\
 &Posterior median &1.0232&1&0.0229\\
 &Posterior mean  &1.0201&1&0.0199\\
 \bottomrule
 \end{tabular}
\end{table}
We first randomly generate $100$ samples from the Pareto distribution with parameters $\alpha_0=1,\beta_0=1$.
The posterior distributions and corresponding estimators are obtained by the methods given in the previous section.
The results of estimators are recorded in Table~\ref{tab:est}.
To illustrate the extent to which the estimators approximate to the underlying parameters,
we show in the last column the distances $d(p_{\alpha_0,\beta_0},p_{\alpha,\beta})$ between the underlying parameters $\alpha_0,\beta_0$ and the estimators of $\alpha,\beta$ by using~\eqref{eqn:dist}.
One interesting phenomenon is that the distances decrease much more in the case of known $\beta$ than in the case of known $\alpha$. This implies that the known $\beta$ conveys much more information than the known $\alpha$.

We also provide graphical representations of the joint posterior and marginal posteriors as follows.
The joint posterior probability density function $p(\alpha,\beta\,|\,\mathbf x)$ is shown in Fig.~\ref{fig:joint}.
As we can see, the joint posterior probability is concentrated in a small area near the MLE $(\widehat{\alpha}(\mathbf x),\widehat{\beta}(\mathbf x))=(1.0303,1.1271)$.
We also note from~\eqref{eqn:post} that the density $p(\alpha,\beta\,|\,\mathbf x)$ is unbounded near $\alpha=0,\beta=0$,
but the inappreciable probability makes it imperceptible from Fig.~\ref{fig:joint}.
\begin{figure}[htbp]
 \centering
  \includegraphics[scale=0.35]{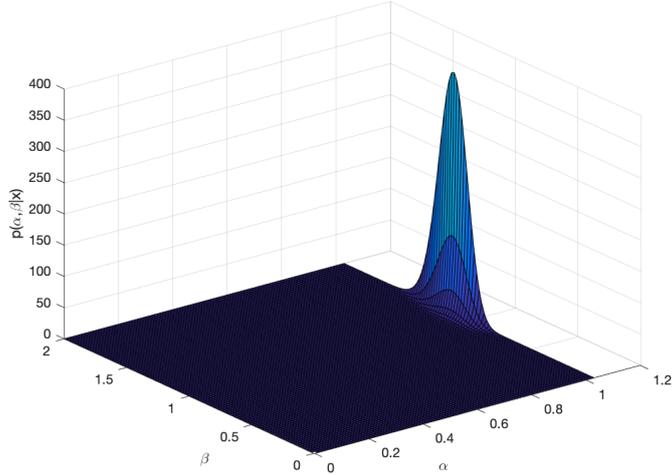}
 \caption{Joint posterior probability density function}
 \label{fig:joint}
\end{figure}

The marginal posterior probability density functions of $\alpha$ and $\beta$ are shown in Fig.~\ref{fig:marg}. From the marginal posterior of $\alpha$ in Fig.~\ref{fig:marg}(\subref{fig:marga}), we see that the probability density $p(\alpha\,|\,\mathbf x)$ is extremely concentrated in the vicinity of $\widehat{\alpha}(\mathbf x)$.
In addition, as we have already discussed,
the density $p(\alpha\,|\,\mathbf x)$ is unbounded near $\alpha=0$.
However, this fact cannot be directly seen from Fig.~\ref{fig:marg}(\subref{fig:marga}), since $\alpha$ diverges so slowly that it needs $\alpha$ being far less than $10^{-280}$ to achieve a noticeable magnitude of $p(\alpha\,|\,\mathbf x)$. Furthermore, by using~\eqref{eqn:cdfa}, we can deduce that the probability for $\alpha<0.9$ is less than $10^{-6}$. Thus, it is very unlikely that $\alpha$ is very distant from $\widehat{\alpha}(\mathbf x)$, let alone being near 0.
Fig.~\ref{fig:marg}(\subref{fig:margb}) illustrates the marginal posterior density $p(\beta\,|\,\mathbf x)$, which approximates to the normal distribution with mean $\widehat{\beta}(\mathbf x)=1.1271$ and standard deviation $\widehat{\beta}(\mathbf x)/\sqrt n=0.1127$, as we have mentioned earlier. This explains the concentration of the probability density of $\beta$ near $\widehat{\beta}(\mathbf x)$ as shown in Fig.~\ref{fig:marg}(\subref{fig:margb}).
 \begin{figure}[htbp]
 \begin{subfigure}[b]{0.5\linewidth}
  \centering
  \includegraphics[scale=0.25]{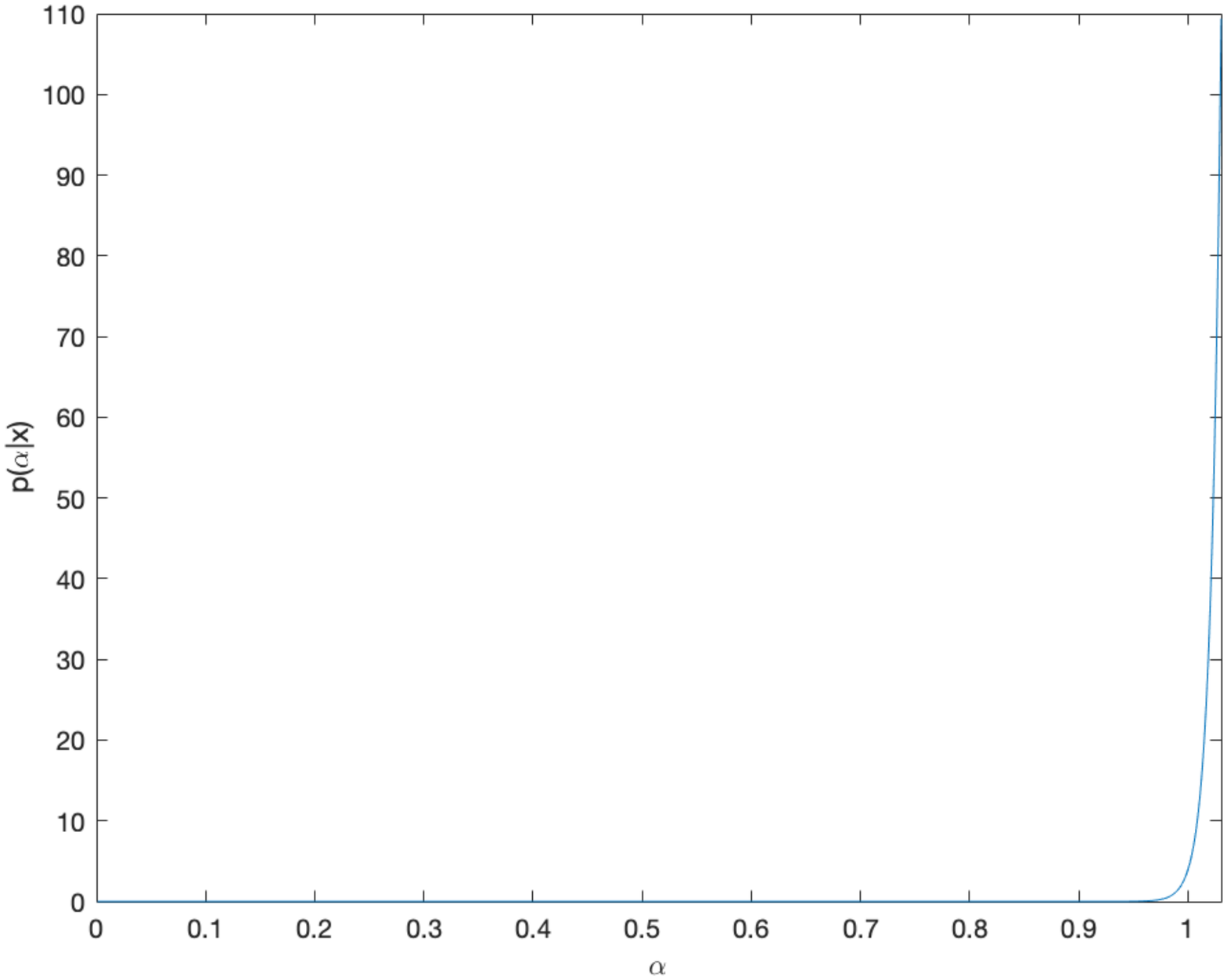}
  \caption{Marginal posterior of $\alpha$}\label{fig:marga}
 \end{subfigure}%
 \begin{subfigure}[b]{0.5\linewidth}
  \centering
   \includegraphics[scale=0.25]{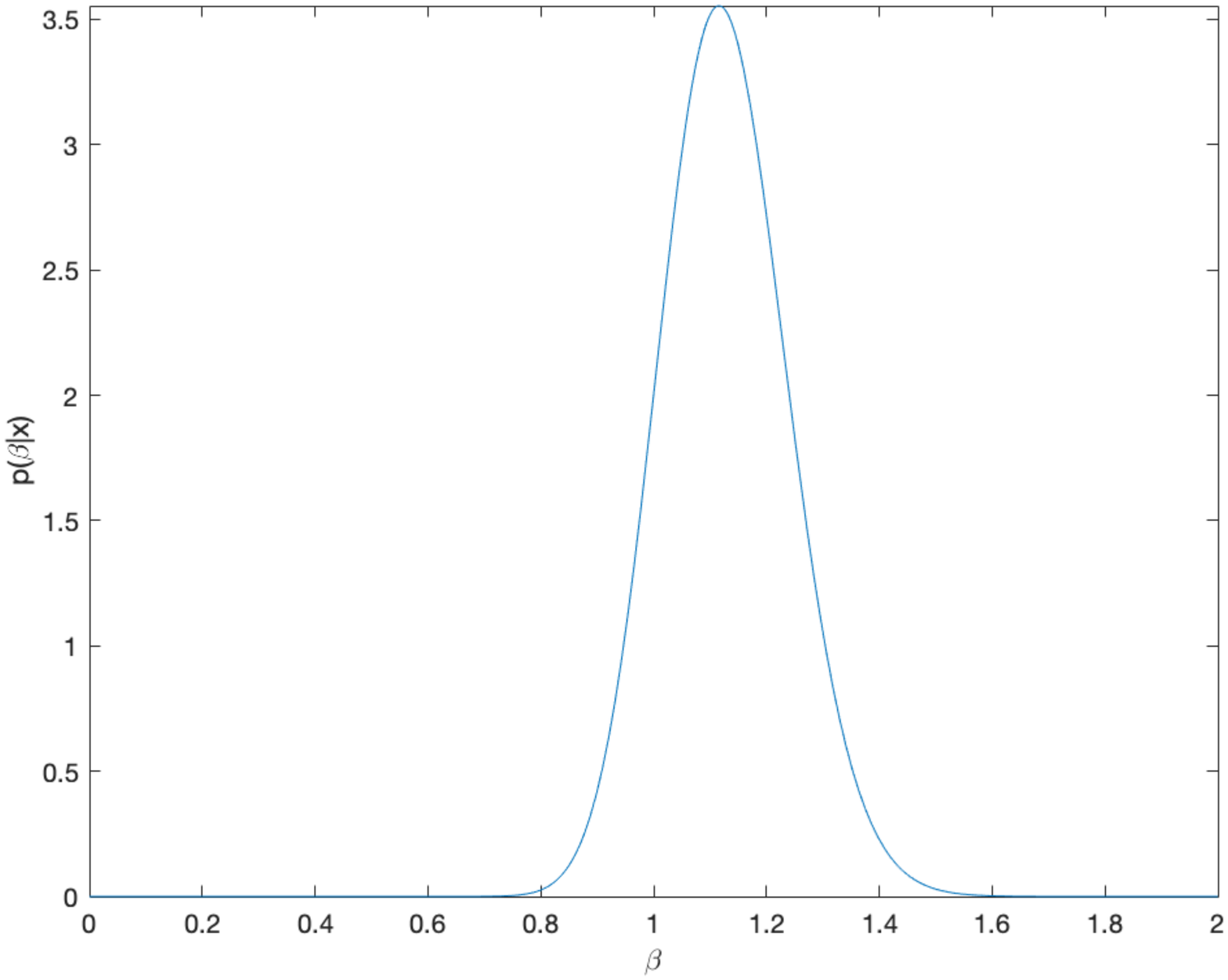}
  \caption{Marginal posterior of $\beta$}\label{fig:margb}
 \end{subfigure}%
\caption{Marginal posterior probability density functions}\label{fig:marg}
\end{figure}

Fig.~\ref{fig:pred} illustrates the approximation to the underlying Pareto distribution by the posterior predictive distribution.
The probability density function $p(x\,|\,\alpha,\beta)$ of the underlying distribution is determined by~\eqref{eqn:pdf} with $\alpha=\alpha_0,\beta=\beta_0$;
the probability density function $p(\widetilde x\,|\,\mathbf x)$ of the posterior predictive distribution is determined by~\eqref{eqn:pred}.
As we can see, the predictive density $p(\widetilde x\,|\,\mathbf x)$ well approximates to the underlying density $p(x\,|\,\alpha,\beta)$, except that it is continuous for $\widetilde x>0$ with a cusp at $\widetilde x=\widehat{\alpha}(\mathbf x)$.
We also note from~\eqref{eqn:pred} that $p(\widetilde x\,|\,\mathbf x)$ is unbounded near $\widetilde x=0$,
which cannot be directly seen from Fig.~\ref{fig:pred}.
\begin{figure}[htbp]
 \centering
  \includegraphics[scale=0.33]{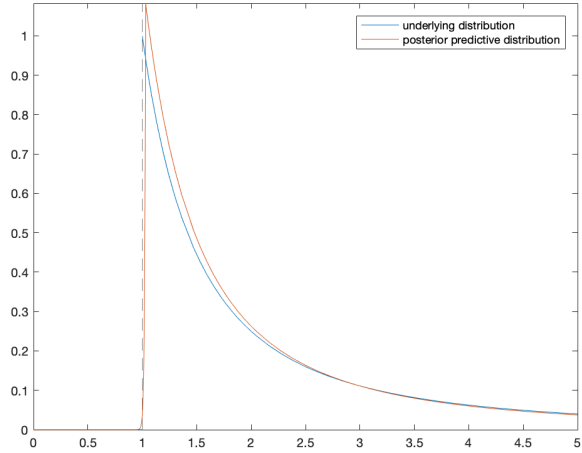}
 \caption{Underlying distribution and posterior predictive distribution}
 \label{fig:pred}
\end{figure}

\section{Conclusion}
In this paper, we proved that the two-parameter family of Pareto distribution with a proper Fisher--Rao metric is isometric to the Poincar\'e upper half-plane model. Geometrical properties of the Pareto distribution, such as connection, curvature and geodesics, could then be studied in the light of the isometry. One notable result is that it contributes as one of the statistical manifolds with constant curvatures among the very few known occasions, including the well-known normal distribution and the Weibull distribution \cite{Cao2008}; a classification of exponential families with constant curvatures is available in \cite{Peng2019}. Jeffreys prior, a non-informative prior closely related to the Fisher--Rao metric, was accordingly obtained to carry out Bayesian inference. We expect that results of this paper would motivate further differential-geometric investigations of statistical manifolds violating the regularity conditions as well as their applications.

\section*{Acknowledgements}
H Sun is supported by the National Natural Science Foundation of China (Nos. 61179031, 10932002). L Peng is supported by the MEXT ``Top Global University Project" and Waseda University Grant for Special Research Projects (Nos. 2019C-179, 2019E-036, 2019R-081).

\bibliographystyle{abbrv}
\bibliography{pareto-ref}

\end{document}